\documentclass[journal]{IEEEtran}
\usepackage{amsthm}
\usepackage{amsmath,amssymb,amsfonts}
\allowdisplaybreaks
\newcommand\numberthis{\addtocounter{equation}{1}\tag{\theequation}}

\usepackage{graphicx}
\usepackage{graphics}
\usepackage{subfigure}

\usepackage{float}

\usepackage{varioref}
\usepackage{hyperref}
\usepackage[noabbrev]{cleveref}

\usepackage{xcolor}

\usepackage{cite}

\ifCLASSINFOpdf
\else
\fi

\newtheorem{remark}{Remark}
\newtheorem{assumption}{Assumption}

\newtheorem{lemma}{Lemma}
\newtheorem{theorem}{Theorem}

\begin{document}
% \title{Distributed Continuous- and Discrete-time Optimization With Event-triggered Communication}
\title{Distributed Optimization With Event-triggered Communication via Input Feedforward Passivity}
\author{Mengmou~Li, Lanlan~Su, and Tao~Liu
\thanks{Liu's work was supported by the Research Grants Council of the Hong Kong Special Administrative Region under the General Research Fund Through Project No. 17256516.}
\thanks{M.~Li and T.~Liu are with the Department
of Electrical and Electronic Engineering, The University of Hong Kong, Hong Kong, China (e-mail: mengmou\_li@hku.hk; taoliu@eee.hku.hk).}

\thanks{L.~Su is with School of Engineering, University of Leicester (e-mail: ls499@leicester.ac.uk).}
}
\date{}
\maketitle
\begin{abstract}
In this work, we address the distributed optimization problem with event-triggered communication by the notion of input feedforward passivity (IFP). First, we analyze the distributed continuous-time algorithm over uniformly jointly strongly connected balanced digraphs in an IFP-based framework.
Then, we propose a distributed event-triggered communication mechanism for this algorithm.
Next, we discretize the continuous-time algorithm by the forward Euler method \textcolor{blue}{with a constant stepsize irrelevant to network size}, and show that the discretization can be seen as a stepsize-dependent passivity degradation of the input feedforward passivity. Thus, the discretized system preserves the IFP property and enables the same event-triggered communication mechanism but without Zeno behavior due to the discrete-time nature. Finally, a numerical example is presented to illustrate our results.
\end{abstract}
\section{Introduction}\label{Introduction}
\IEEEPARstart{D}{istributed} optimization problem aims to optimize the sum of objective functions of agents cooperatively, where each agent estimates the optimal solution based on local information and information obtained from its neighbors through a communication network. It has been widely studied in recent years, and numerous algorithms have been proposed, which can be categorized into two groups, i.e., discrete-time algorithms \cite{nedic2017achieving,xie2018distributed,li2018distributed,scutari2019distributed} and continuous-time algorithms \cite{gharesifard2013distributed,kia2015distributed,li2017distributed,li2019generalized}.

An important issue in distributed optimization is the relaxation of communication graph conditions, since the communication network may be unidirectional or even time-varying in practice. The work \cite{gharesifard2013distributed} generalizes the well-known proportional-integral (PI) algorithm to weight-balanced and strongly connected digraphs. A fully distributed adaptive algorithm for the design of parameters is proposed in \cite{li2017distributed}.
The work \cite{kia2015distributed} proposes a modified PI algorithm over weight-balanced and strongly connected switching digraphs. Recently, \cite{touri2019modified} incorporates a continuous-time push-sum algorithm to address the problem of general directed graphs. Besides, \cite{li2019input} proposes fully distributed algorithms over weight-balanced and uniformly jointly strongly connected digraphs based on input feedforward passivity (IFP).
There are also many related works in the discrete-time setting (e.g., \cite{nedic2017achieving,xie2018distributed,li2018distributed,scutari2019distributed} and references therein), while most of them adopt diminishing stepsizes or require global information.

Another critical problem for distributed optimization is event-triggered based communication motivated by practical issues in the communication network such as network congestion, limited bandwidth, and energy consumption.
A centralized event-triggered strategy is proposed in \cite{kia2015distributed}. An encoder-decoder event-trigger communication mechanism is introduced in \cite{liu2016event}. An edge-based event-triggered method is proposed in \cite{kajiyama2018distributed}.
The work \cite{liu2019distributed} considers event-triggered communication for distributed optimization with geometric constraints.
% For more related works readers can refer to \cite{kia2015distributed,liu2016event,kajiyama2018distributed} and references therein.
However, the abovementioned works only consider cases with undirected or fixed communication networks and cannot apply to directed and switching networks. Recently, a periodic sampling communication mechanism is proposed in \cite{su2019distributed} over weight-balanced and uniformly jointly strongly connected digraphs for resource allocation problems.
% {\color{red}What are the differences between the current paper and \cite{su2019distributed}?} {\color{blue} The event-triggered part in [16] is removed in the second submission by comments from reviewers. In addition, we address the exponential convergence and discretization in this manuscript.}
Considering distributed algorithms under event-triggered control over uniformly jointly strongly connected digraphs is of great significance, since the communication effort can be greatly reduced due to the lack of graph connectivity or consecutive communication. To the best of our knowledge, this problem has not been addressed yet.

In this work, we consider a distributed optimization problem over uniformly jointly strongly connected balanced graphs and propose a distributed event-triggered communication scheme for both the continuous-time and discrete-time algorithms via IFP.
The continuous-time algorithm has been proposed in \cite{li2019input,kia2015distributed}, while its discrete-time counterpart is first introduced here.

% The rest of this letter is organized as follows. \Cref{Preliminaries} give the preliminaries and problem formulation. \Cref{Continuous-time Algorithm} discuss continuous-time algorithm along with event-triggered mechanism while its descretization is addressed in \Cref{Discrete-time Algorithm}. A numerical example is provided in \Cref{Numerical Example} and finally \Cref{Conclusion} concludes the letter.

% In this work, we address the distributed optimization problem by exploiting the notion of IFP. First, we analyze a distributed continuous-time algorithm over uniformly jointly strongly connected balanced digraphs and show its exponential convergence over strongly connected digraphs. Then, we propose an event-triggered communication mechanism for this algorithm based on the IFP. Next, we discretize the continuous-time algorithm by the forward Euler method and show that the discretization can be seen as a stepsize dependent passivity degradation of the IFP. The discretized system preserves IFP property and enables the same event-triggered communication mechanism but without Zeno behavior due to the {\color{blue}discrete-time nature}.

\section{Preliminaries}\label{Preliminaries}
\subsection{Notation}\label{Notation}
Let $\mathbb{R}$ be the set of real numbers. The Kronecker product is denoted as $\otimes$. Let $\left\| \cdot \right\|$ denote the 2-norm of a vector and also the induced 2-norm of a matrix.
Given a symmetric matrix $M \in \mathbb{R}^{m\times m}$, $M>0$ $(M \geq 0)$ means that $M$ is positive definite (positive semi-definite).
% Denote the eigenvalues of $M$ as $s_1(M) \leq s_2(M) \leq \ldots \leq s_m(M)$.
Let $I$ and $\mathbf{0}$ denote the identity matrix and zero matrix with proper dimensions, respectively. $\mathbf{1}_m := (1,\ldots, 1)^T \in \mathbb{R}^{m}$ denotes the vector with all ones.
% We denote a vector $v$ without subscript as $v = \text{col}(v_1,\ldots, v_m) := (v_1^T, \ldots, v_m^T)^T$ as a compact vector of vectors $v_1, \ldots, v_m$ if not otherwise specified.

\subsection{Convex Analysis}
A differentiable function $f: \mathbb{R}^{m} \rightarrow \mathbb{R}$ is \textit{convex} over a convex set $\mathcal{X} \subset \mathbb{R}^{m}$ if and only if $\left(\nabla f(x)-\nabla f(y)\right)^{T}(x-y)\geq 0$, $\forall x,~y \in \mathcal{X}$.
% , and is \textit{strictly convex} if and only if this strict inequality holds for any $x \neq y$. 
It is \textit{$\mu$-strongly convex} if and only if $\left(\nabla f(x)-\nabla f(y)\right)^{T}(x-y)\geq \mu \left\|x-y \right\|^{2}$, or equivalently,
$f(y) \geq f(x) + \nabla f(x)^T (y - x) + \frac{\mu}{2} \left\|y - x \right\|^2$, $\forall x, y \in \mathcal{X}$.
An operator $\mathbf{f}: \mathbb{R}^{m} \rightarrow \mathbb{R}^{m}$ is \textit{$l$-Lipschitz continuous} over a set $\mathcal{X} \in \mathbb{R}^{m}$ if $\left\| \mathbf{f}(x) - \mathbf{f}(y) \right\| \leq l \left\| x - y \right\|$, $\forall x,y\in \mathcal{X}$.

\subsection{Communication Network and Graph Theory}
The communication network is represented by a graph $\mathcal{G} = (\mathcal{N},\mathcal{E})$, where $\mathcal{N} = \{1,\ldots,N\}$ is the node or agent set, $\mathcal{E}\subset \mathcal{N}\times\mathcal{N}$ is the edge set. The edge $(i,j) \in \mathcal{E}$ means that agent $i$ can send information to agent $j$.
The graph $\mathcal{G}$ is said to be \textit{undirected} if $(i,j) \in \mathcal{E} \Leftrightarrow (j,i) \in \mathcal{E}$ and \textit{directed} otherwise.
The adjacency matrix $A = [a_{ij}]$ of $\mathcal{G}$ is defined as $a_{ii} = 0$; $a_{ij} > 0$ if $(j,i) \in \mathcal{E}$, and $a_{ij} = 0$, otherwise.
$\mathcal{G}$ is said to be \textit{strongly connected} if there exists a sequence of successive edges between any two agents. The in-degree and out-degree of the $i$th agent are $d_{in}^{i} = \sum_{j=1}^{N} a_{ij}$ and $d_{out}^{i} = \sum_{j=1}^{N} a_{ji}$, respectively. The graph $\mathcal{G}$ is said to be \textit{weight-balanced} if $d_{in}^{i} = d_{out}^{i},~\forall i \in \mathcal{N}$.
The Laplacian matrix of $\mathcal{G}$ is defined as $L = \textrm{diag}\{A \mathbf{1}_{N}\} - A$. Clearly, $L \mathbf{1}_{N} = \mathbf{0}$. Moreover, if $\mathcal{G}$ is weight-balanced, then $\mathbf{1}_{N}^T L = \mathbf{0}$. A time-varying graph $\mathcal{G}(t)$ with fixed nodes is said to be \textit{uniformly jointly strongly connected} (UJSC) if there exists a $T > 0$ such that for any $t_k$, the union $\cup_{t \in [t_k, t_k + T]} \mathcal{G}(t)$ is strongly connected.

\subsection{Passivity}
 Consider a nonlinear system $\Sigma$ described by
 \begin{equation}\label{general nonlinear system}
 \begin{cases}
 x^{+} = F \left(x, u \right)\\
 y = H \left( x,u \right)
 \end{cases}
 \end{equation}
% \begin{equation}\label{general nonlinear system}
%  \Sigma : ~
%  x^{+} = F \left(x, u \right);\quad y = H \left( x,u \right)
% \end{equation}
where $x \in \mathcal{X} \subset \mathbb{R}^n$, $u \in \mathcal{U} \subset \mathbb{R}^m$ and $y \in \mathcal{Y} \subset \mathbb{R}^m$ are the state, input and output, respectively, and $\mathcal{X}$, $\mathcal{U}$ and $\mathcal{Y}$ are the state, input and output spaces, respectively. $x^{+}$ denotes the derivative of the state in the continuous-time (CT) case and the state at the next time step in the discrete-time (DT) case. The nonlinear functions $F: \mathcal{X} \times \mathcal{U} \rightarrow \mathbb{R}^{n}$, $H: \mathcal{X} \times \mathcal{U} \rightarrow \mathcal{Y}$ are assumed to be sufficiently smooth.

System $\Sigma$ is said to be \textit{passive} if there exists a continuously differentiable positive semi-definite function $V(x)$, called the \textit{storage function}, such that
$
\dot{V}(x) \leq u^T y,~\forall x\in \mathcal{X},~u\in \mathcal{U}
$ in CT case (or, $V\left( x(k+1) \right) - V\left( x(k) \right)\leq u^T y,~\forall x\in \mathcal{X},~u\in \mathcal{U},$ in DT case).
Moreover, it is said to be \textit{input feedforward passive} (IFP) if $\dot{V}(x) \leq u^T y - \nu u^T u$ for CT case (or, $V\left( x(k+1) \right) - V\left( x(k) \right)\leq u^T y - \nu u^T u$ for DT case), for some $\nu \in \mathbb{R}$, denoted as IFP($\nu$).
% \begin{definition}
% System $\Sigma$ is said to be \textit{passive} if there exists a continuously differentiable positive semi-definite function $V(x)$, called the \textit{storage function}, such that
% \begin{equation}\label{passivity definition inequality}
% \dot{V}(x) \leq u^T y, \quad \forall x,~u
% \end{equation}
% Moreover, it is said to be \textit{input feedforward passive} (IFP) if $\dot{V}(x) \leq u^T y - \nu u^T u$, for some $\nu \in \mathbb{R}$, denoted as IFP($\nu$).
% \end{definition}
The sign of the IFP index $\nu$ denotes an excess or shortage of passivity. 

% For a discrete-time system, {\color{red} (it is better to give the system model here)} {\color{blue}It would be perfect. But it is also not easy to make space due to page limit-Mike.} the definitions are similar with $\dot{V}(x)$ replaced by $V\left( x(k+1) \right) - V\left( x(k) \right)$.

% In particular, when $\nu > 0$, the system is said to be \textit{input strictly passive} (ISP). When $\nu < 0$, the system is said to be \textit{input feedforward passivity-short} (IFPS).

\subsection{Problem Formulation}
Consider a distributed optimization problem among a group of agents
\begin{align}\label{eq: distributed convex optimization}
% \[
% \begin{array}{rl}
\min_{\mathrm{x}} \sum_{i=1}^{N} f_i(\mathrm{x})
% \numberthis \label{eq: distributed convex optimization}
% \end{array}
% \]
\end{align}
where $f_i : \mathbb{R}^{m} \rightarrow \mathbb{R}$ is the local objective function for agent $i$, $\forall i \in \mathcal{N} = \{1,\ldots,N\}$ and $\mathrm{x} \in \mathbb{R}^{m}$ is the decision variable.
We consider problem \eqref{eq: distributed convex optimization} with the following assumptions.

\begin{assumption}\label{Ass: objective function}
Each local function $f_i$ is sufficiently smooth, $\mu_i$-strongly convex and has $l_i$-Lipschitz continuous gradient.
\end{assumption}

\begin{assumption}\label{Ass: graph}
The time-varying communication digraph $\mathcal{G}$ is weight-balanced pointwise in time and UJSC.
% with bounded Laplacian matrix
\end{assumption}
\begin{assumption}\label{Ass: graph 2}
The communication protocol is designed such that both the sender and receiver of an edge are aware of its existence.
\end{assumption}

\begin{remark}
\Cref{Ass: graph 2} is a standard assumption in the literature, where each agent knows its out-degree \cite{nedic2017achieving, scutari2019distributed}. Then, each agent can locally manipulate its in/out-degree to render the global graph weight-balanced, while the strong connectedness is not required pointwise in time.
Moreover, no restriction on the switching rules
% \textcolor{red}{(switching setting?)} 
is imposed. The graph can change continuously provided \Cref{Ass: graph} and \ref{Ass: graph 2} are satisfied.
\end{remark}
% (\textcolor{red}{Does this comment arouse the problem how each agent's manipulation ensure global UJSC?})
% \textcolor{red}{Does this sentence repeat Assumption 2?}

\section{Continuous-time Algorithm}\label{Continuous-time Algorithm}
\subsection{IFP-based Distributed Algorithm}
We adopt a distributed algorithm that consists of a group of input feedforward passive system in the following form to solve problem \eqref{eq: distributed convex optimization},
\begin{subequations}\label{eq: algorithm individual agents}
\begin{align}
\dot{x}_i = & -\alpha \nabla f_i(x_i) - \lambda_i \label{eq: algorithm individual agents dynamics}\\
\dot{\lambda}_i = & - u_i \label{eq: algorithm individual agents input}
\end{align}
\end{subequations}
where $x_i \in \mathbb{R}^{m}$ is the decision variable for agent $i$, $\lambda_i \in \mathbb{R}^{m}$ is an auxiliary state for agent $i$ to track the difference between neighboring agents.
In the above individual system, $u_i$ is the system input taking the diffusive couplings of $x_i$, i.e.,
\begin{align}\label{eq: input u_i}
u_i = & \beta \sum_{j=1}^{N} a_{ij}(t)(x_j - x_i).
\end{align}
% or $u = - (L \otimes I) x$ compactly, where $u = (u_1^T, \ldots, u_N^T)^T$ and $x = (x_1^T, \ldots, x_N^T)^T$.
In the above equations, $\alpha > 0$, $\beta > 0$ are parameters to be designed. To ensure optimality under input \eqref{eq: input u_i}, $\lambda_i$ should satisfy $\sum_{i = 1}^{N} \lambda_i(0) = \mathbf{0}$, which can be fulfilled by setting $\lambda_i(0) = 0$, for all $i \in \mathcal{N}$.

% The compact form of \eqref{eq: algorithm individual agents}, \eqref{eq: input u_i} is written as
% \begin{equation}\label{eq: algorithm}
% \begin{aligned}
% \dot{x} &= -\alpha \nabla f(x) - \lambda\\
% \dot{\lambda} &= \beta \mathbf{L}(t) x
% \end{aligned}
% \end{equation}
% where $\mathbf{L}(t) = L(t) \otimes I_{N}$ and $L(t)$ is the Laplacian matrix.

Note that algorithm \eqref{eq: algorithm individual agents} is a simplified version of the algorithm reported in \cite{kia2015distributed,li2019input} in which \eqref{eq: algorithm individual agents dynamics} becomes $\dot{x}_i = -\alpha \nabla f_i(x_i) - \lambda_i + u_i$.
% (\textcolor{red}{Maybe remove "without coupling terms in (3a)" since the readers don't know what that is.})
In this work, we first analyze the algorithm \eqref{eq: algorithm individual agents} with \eqref{eq: input u_i} in an input-feedforward passivity-based framework, and then propose an event-triggered mechanism for the algorithmic dynamics.

Define $\left(x_i^*, \lambda_i^*\right)$ as the equilibrium point of system \eqref{eq: algorithm individual agents}. 
Then, equilibrium point in \eqref{eq: algorithm individual agents input} ensures that $u_i \equiv 0$, which further gives $x_i^{*} = x_j^{*}$, $\forall i, j \in \mathcal{N}$. Summing up \eqref{eq: algorithm individual agents dynamics}, $\forall i \in \mathcal{N}$, one has
% \begin{align}\label{eq: summing up}
\[
\begin{array}{rl}
-\alpha \displaystyle\sum_{i = 1}^{N} \nabla f_i(x_i^*) = \sum_{i = 1}^{N} \lambda_i^* = (\mathbf{1}_{N}^T \otimes I)\lambda^* = \sum_{i = 1}^{N} \lambda_i(0) = \mathbf{0} \numberthis \label{eq: summing up}
\end{array}
\]
% \end{align}
with $\lambda^* = ({\lambda_1^*}^T, \ldots, {\lambda_N^*}^T)^T$. The third equality in the above equation follows from $\mathbf{1}_{N}^T L = \mathbf{0}$.
Then $x_i^{*}$ is the unique optimal solution to problem \eqref{eq: distributed convex optimization} (see \cite{li2019input} for details).
% Obviously, $x_i^{*} = x_j^{*}$, $\forall i, j \in \mathcal{N}$, with $x_i^{*}$ being the optimal solution to problem \eqref{eq: distributed convex optimization} \cite{li2019input}.
% Moreover, by \Cref{Ass: objective function}, $\left(x_i^*, \lambda_i^*\right)$ is unique.
Denote $\Delta x_i = x_i - x_i^*$, $\Delta \lambda_i = \lambda_i - \lambda_i^*$, and 
% The error system from the optimal point is
% \begin{equation}\label{eq: error subsystem}
% \begin{aligned}
% \Delta \dot{x}_i =& - \alpha \left( \nabla f_i(x_i) - \nabla f_i(x_i^*) \right) - \Delta \lambda_i\\
% \Delta \dot{\lambda}_i =& - u_i\\
% u_i =& \beta\sum_{j=1}^{N} a_{ij}(t)(\Delta x_j - \Delta x_i)
% \end{aligned}
% \end{equation}
\begin{align}\label{eq: Bx}
\nabla f_i(x_i)-\nabla f_i(x_i^{*})=B_{x_i}\Delta x_i
\end{align}
where $B_{x_i}$ is defined as $B_{x_{i}}=\int_{0}^{1} \nabla ^{2}f_i(x_i^{*}+\tau (x_i-x_i^{*})) d \tau$. It follows that 
$\Delta \dot{x}_i =-\alpha B_{x_{i}} \Delta x_i-\Delta \lambda_i.$
Under Assumption \ref{Ass: objective function}, we have $\mu_i I\le B_{x_{i}} \le l_i I$.
\begin{lemma}\label{Lemma error system IFP}
Under Assumption \ref{Ass: objective function}, system \eqref{eq: algorithm individual agents} is IFP($\nu_i$) from $u_i$ to $\Delta x_i$ with index $\nu_i = -\frac{1}{\alpha^2 \mu_i^2}$ with respect to the storage function
% \begin{equation}\label{eq: Storage function}
% \begin{aligned}
\[
\begin{array}{rl}
V_i = & \frac{1}{\alpha \mu_i} \left\| \Delta \dot{x}_i \right\|^2 - \Delta x_i^T \Delta \lambda_i + \alpha \nabla f_i(x_i^*)^T \Delta x_i\\
& + \alpha \left( f_i(x_i^*) - f_i(x_i) \right).\numberthis \label{eq: Storage function}
\end{array}
\]
% \end{aligned}
% \end{equation}
% Moreover, $V_i$ is radially unbounded.
% and there exists a constant $\underline{\epsilon}_i > 0$ such that 
% \begin{align}\label{eq: radially unbounded of Vi}
% \[
% \begin{array}{rl}
% \underline{\epsilon}_i \left\| 
% \left( \begin{smallmatrix}
% \alpha B_{x_{i}} \Delta x_i \\ \Delta \lambda_i
% \end{smallmatrix}\right)\right\|^2 \leq V_i \leq \frac{2}{\alpha \mu_i}\left\|\left(
% \begin{smallmatrix}
% \alpha B_{x_{i}} \Delta x_i \\ \Delta \lambda_i
% \end{smallmatrix}\right)\right\|^2. \numberthis \label{eq: radially unbounded of Vi}
% \end{array}
%  \]
% \end{align}
\end{lemma}

\begin{proof}
% To prove \eqref{eq: radially unbounded of Vi}, 
The strong convexity of $f_{i}(x_i)$ provides that 
% \[
% \begin{array}{r}
% f_{i}(x_{i}^{*})-f_{i}(x_{i})\ge-\nabla f_i(x_{i})^{T}\Delta x_{i}+\frac{\mu_i}{2}\left\| \Delta x_{i}\right\|^{2},\\
% f_{i}(x_{i}^{*})-f_{i}(x_{i})\le -\nabla f_i(x_{i}^{*})^{T}\Delta x_{i}-\frac{\mu_i}{2}\left\| \Delta x_{i}\right\|^{2},
% \end{array}
% \]
% which follows that 
\[
\begin{array}{rl}
 & \alpha \left( f_i(x_i^*) - f_i(x_i) \right) + \alpha \nabla f_i(x_i^*)^T \Delta x_i\\
\ge & \Delta x_{i}^{T}\left(-\alpha B_{x_i}+\frac{\alpha \mu_i}{2}I\right)\Delta x_{i}.
\end{array}
\]
% and 
% $
% \alpha \left( f_i(x_i^*) - f_i(x_i) \right) + \alpha \nabla f_i(x_i^*)^T \Delta x_i
% \le -\frac{\alpha\mu_i}{2} \left\|\Delta x_{i}\right\|^{2}.
% $
Therefore, it can be derived that 
\[
V_i\ge 
\left(\begin{smallmatrix} \alpha B_{x_{i}} \Delta x_i \\ \Delta \lambda_i
\end{smallmatrix} \right)^{T}
% \underline
{R}_i
\left(\begin{smallmatrix}\alpha B_{x_{i}} \Delta x_i \\ \Delta \lambda_i
\end{smallmatrix} \right)
\]
where $%\underline
{R}_i = \left(\begin{smallmatrix}
\frac{I}{\alpha \mu_i} - \frac{1}{\alpha}B_{x_i}^{-1} +\frac{\mu_i}{2 \alpha}B_{x_i}^{-2} & \frac{I}{\alpha \mu_i}-\frac{1}{2\alpha}B_{x_i}^{-1}\\
 \frac{I}{\alpha\mu_i}-\frac{1}{2\alpha}B_{x_i}^{-1} & \frac{I}{\alpha \mu_i}
\end{smallmatrix}\right) > 0$. Thus, $V_i$ is positive definite and radially unbounded with respect to $\left\| \left(\begin{smallmatrix} \alpha B_{x_{i}} \Delta x_i \\ \Delta \lambda_i
\end{smallmatrix} \right) \right\|$.
% and
% \[
% V_i\le \left(\begin{smallmatrix}\alpha B_{x_{i}} \Delta x_i \\ \Delta \lambda_i
% \end{smallmatrix} \right)^{T}
% \bar{R}_i
% \left(\begin{smallmatrix}\alpha B_{x_{i}} \Delta x_i \\ \Delta \lambda_i
% \end{smallmatrix} \right)
% \]
% where $\bar{R}_i = \left(\begin{smallmatrix}
% \frac{I}{\alpha \mu_i}-\frac{\mu_i}{2 \alpha}B_{x_i}^{-2} & \frac{I}{\alpha \mu_i}-\frac{1}{2\alpha}B_{x_i}^{-1}\\
%  \frac{I}{\alpha\mu_i}-\frac{1}{2\alpha}B_{x_i}^{-1} & \frac{I}{\alpha \mu_i}
% \end{smallmatrix} \right) > 0$.
% By some calculations, we obtain that $
% \underline{R}_i \ge \underline{\epsilon}_i I = \frac{1}{\alpha \mu_i} \left(\frac{5}{4} - \frac{ \mu_i}{2 l_i} - \sqrt{\frac{25}{16} - \frac{5 \mu_i}{4 l_i}} \right)I > 0$, and $\bar{R}_i < \frac{2}{\alpha \mu_i} I$. Hence, \eqref{eq: radially unbounded of Vi} is verified. 
% Next, we show that with the storage function $V_{i}$, the system \eqref{eq: algorithm individual agents} is IFP($\nu_{i}$) from
% $u_{i}$ to $\Delta x_{i}$. 
% It can be obtained that 
Taking the derivative of $V_i$ along system \eqref{eq: algorithm individual agents} gives 
% \begin{eqnarray}
\[
\begin{array}{rl}
\dot{V_{i}} = & \frac{1}{\alpha \mu_i} \frac{d \left\| \Delta\dot{x}_{i} \right\|^{2}}{dt}+\frac{d(-\Delta x_{i}^{T}\Delta\lambda_{i})}{dt}+\nonumber \\
 & \alpha\cdot\frac{d\left(f_{i}(x_{i}^{*})-f_{i}(x_{i})+\nabla f_i(x_i^{*})^{T}\Delta x_{i}\right)}{dt}\nonumber \\
\le & -2\left\| \Delta\dot{x}_{i} \right\|^{2} + \frac{2}{\alpha \mu_i}\Delta\dot{x}_{i}^{T}u_{i}+\Delta x_{i}^{T}u_{i}\nonumber \\
& -\left(\alpha B_{x_{i}}\Delta x_{i}+\Delta\lambda_{i}\right)^{T}\Delta\dot{x}_{i}\nonumber \\
 = & -||\Delta\dot{x}_{i}||^{2}+\frac{2}{\alpha \mu_i}\Delta\dot{x}_{i}^{T}u_{i}+\Delta x_{i}^{T}u_{i} \nonumber\\
\le & \Delta x_{i}^{T}u_{i}+ \frac{1}{\alpha^2 \mu_i^2}||u_i||^{2}, \numberthis
\end{array}
 \]
% \end{eqnarray}
% Thus, it can be summarized that $\dot{V_{i}}\le\Delta x_{i}^{T}u_{i}+\frac{1}{\alpha^{2} \mu_{i}^2}||u_{i}||^{2}$,
which completes the proof. 
\end{proof}
%The proof is straightforward, readers can also refer to \cite{li2019input} for more details. Besides, we can obtain that $V_i$ is radially unbounded and satisfies 
%\begin{align}\label{eq: radially unbounded %of Vi}
%\frac{1}{\alpha \mu_i} \left\|
%\begin{smallmatrix}
%\nabla f_i(x_i) - \nabla f_i(x_i^*) \\ %\Delta \lambda_i
%\end{smallmatrix}\right\|^2 \leq V_i \leq %\frac{4}{\alpha \mu_i}\left\|
%\begin{smallmatrix}
%\nabla f_i(x_i) - \nabla f_i(x_i^*) \\ %\Delta \lambda_i
%\end{smallmatrix}\right\|^2.
%\end{align}
% By Lipschitz continuity, we have
% $\underline{\epsilon} \left\|
% \begin{smallmatrix}
% \Delta x_i \\ \Delta \lambda_i
% \end{smallmatrix}\right\|^2 \leq V_i \leq \bar{\epsilon}\left\|
% \begin{smallmatrix}
% \Delta x_i \\ \Delta \lambda_i
% \end{smallmatrix}\right\|^2$.
% where $\underline{\epsilon} = \min\{\frac{1}{\alpha \mu_i}, \frac{\alpha l_i^2}{\mu_i} \}$, and $\bar{\epsilon} = \min\{\frac{4}{\alpha \mu_i}, \frac{ 4\alpha l_i^2}{\mu_i} \}$.

\begin{lemma}[\hspace{1sp}\cite{li2019input}]\label{Lemma convergence of IFP}
Under Assumption \ref{Ass: objective function} and \ref{Ass: graph}, the states of algorithm \eqref{eq: algorithm individual agents} with initial condition $\sum_{i=1}^{N} \lambda_i(0)= \mathbf{0}$ will converge to the optimal solution to problem \eqref{eq: distributed convex optimization} if the following condition holds,
\begin{equation}\label{eq: individual condition}
|\nu_i| \beta d_{in}^{i}(t) < \frac{1}{2},~\forall i\in\mathcal{N},~\forall t \geq 0
\end{equation}
where $d_{in}^{i}(t)$ denotes the in-degree of the $i$th agent and $\nu_i$ is the IFP index defined in \Cref{Lemma error system IFP}.
\end{lemma}
% It can be proved by adopting $V = \sum_{i=1}^{N} V_i$ as the Lyapunov function candidate.
This lemma characterizes the design of parameters $\alpha$, $\beta$ through input-feedforward passivity. In practice, we can fix one variable and design the other one.
\Cref{Lemma convergence of IFP} can be proved by considering the Lyapunov function candidate $V = \sum_{i = 1}^{N} V_i$. Readers can refer to \cite{li2019input} for the proof.

\subsection{Event-triggered Mechanism}
In this subsection, we reconsider the algorithm in \eqref{eq: algorithm individual agents} by incorporating an event-triggered communication mechanism, i.e., instead of transmitting the real-time $x_i$, an event-triggered-based input is considered,
\begin{equation}\label{eq: input with event-triggered}
\begin{aligned}
u_i = & \beta\sum_{i=1}^{N}a_{ij}(t)(\hat{x}_{j}-\hat{x}_{i})
\end{aligned}
\end{equation}
where $\hat{x}_{i},\ i\in\mathcal{N}$ denotes the latest sampled state of agent $i$ that has been transmitted to its neighbors and $\hat{x}_i(0) := x_i(0)$.
In \eqref{eq: input with event-triggered}, each agent only updates its current state $x_i$ to its out-neighbors when the local error signal $e_{i}(t) = x_{i}(t)-\hat{x}_{i}(t)$ exceeds a threshold depending on the latest received state of $x_j$ from its in-neighbors.
In this work, the triggering condition is 
\begin{equation}\label{eq: event-driven condition}
\left\|e_{i}(t) \right\|^2 \ge \frac{c_i}{d^{i}_{in}(t)}\left(\frac{1}{2} - |\nu_i| \beta d_{in}^{i}(t)\right)^2 \sum_{j=1}^{N} a_{ij}(t)\left\| \hat{x}_{j}-\hat{x}_{i}\right\|^{2}
\end{equation}
where $ c_i \in (0,1)$ is a constant. This triggering condition is fully distributed since only local information is needed.

% \begin{remark}
In a time-varying graph, we stipulate that, whenever a link between two agents appears, the sender sends its last triggered state to the receiver, which is not considered as a ``triggering''.
Whenever a link disappears, the receiver modifies \eqref{eq: input with event-triggered} accordingly such that the disconnection between agents is not confused with the ``connected but non-triggering'' case. This can be guaranteed by \Cref{Ass: graph 2}.
% \end{remark}

The following theorem presents the convergence to the global optimal solution under event-triggered communication. 

\begin{theorem}\label{convergence under ETC continuous-time}
Under Assumptions \ref{Ass: objective function}--\ref{Ass: graph 2}, if $\alpha,\beta$ are designed such that \eqref{eq: individual condition} holds, and the triggering instant for agent $i$, $i\in \mathcal{N}$ to transmit its
current information of $x_{i}$ is chosen whenever $d^i_{in}(t) >0$ and the triggering condition \eqref{eq: event-driven condition} is satisfied.
{Suppose there exists a solution to system \eqref{eq: algorithm individual agents} under event-triggered control \eqref{eq: input with event-triggered}, \eqref{eq: event-driven condition} for all $t \geq 0$.}
Then the states with initial condition $\sum_{i=1}^{N}\lambda_{i}(0)=\mathbf{0}$
will converge to the optimal solution to problem \eqref{eq: distributed convex optimization}.
\end{theorem}
\begin{proof}
Consider the Lyapunov function candidate $V=\sum_{i=1}^{N}V_{i} \geq 0$, where $V_i$ was defined in \eqref{eq: Storage function}. From \Cref{Lemma error system IFP}, its derivative along \eqref{eq: algorithm individual agents} and \eqref{eq: input with event-triggered} yields
\begin{align*}
% \[
% \begin{array}{rl}
\dot{V}\le & { \displaystyle\sum_{i=1}^{N}}\Delta x_{i}^{T} u_{i} - \nu_i u_{i}^{T}u_{i} \\
= & { \displaystyle\sum_{i=1}^{N}}\left(\beta\Delta x_{i}^{T} \displaystyle\sum_{j=1}^{N}a_{ij}(t)\left(\Delta\hat{x}_{j}-\Delta\hat{x}_{i}\right)\right)\\
& +{ \displaystyle\sum_{i=1}^{N}} \beta^2 |\nu_i| \bigg\| \displaystyle\sum_{j=1}^{N} a_{ij}(t)\left(\Delta\hat{x}_{j}-\Delta\hat{x}_{i}\right) \bigg\|^{2}\\
= & { \displaystyle\sum_{i=1}^{N}}\left(\beta\left(\Delta\hat{x}_{i}+e_{i}\right)^{T}{ \displaystyle\sum_{j=1}^{N}} a_{ij}(t)\left( \Delta\hat{x}_{j}-\Delta\hat{x}_{i}\right)\right)\\
 & + { \displaystyle\sum_{i=1}^{N}} \beta^{2}|\nu_i| \bigg\| { \displaystyle\sum_{j=1}^{N}}a_{ij}(t)\left( \hat{x}_{j}- \hat{x}_{i}\right) \bigg\|^{2}\\
= & \beta {\displaystyle \sum_{i=1}^{N}} \left( {\displaystyle \sum_{j=1}^{N}}e_{i}^{T}a_{ij}(t)\left( \hat{x}_{j}- \hat{x}_{i}\right)+{\displaystyle\sum_{j=1}^{N}}a_{ij}(t)\Delta\hat{x}_{i}^{T}\Delta\hat{x}_{j}\right.\\
 & \left.-{\displaystyle \sum_{j=1}^{N}}a_{ij}(t)\Delta\hat{x}_{i}^{T}\Delta\hat{x}_{i} + \beta |\nu_i| \bigg\| {\displaystyle \sum_{j=1}^{N}}a_{ij}(t)\left( \hat{x}_{j}- \hat{x}_{i}\right) \bigg\|^{2} \right)
% \end{array}
% \]
\end{align*}
where {$\Delta \hat{x}_i : = \hat{x}_i - x_i^*$}. The second equality in the above equation holds since $e_{i}=x_{i}-\hat{x}_{i}=\Delta x_{i} - \Delta\hat{x}_{i}$.
Observe that 
% \begin{align*}
\[
\begin{array}{rl}
e_{i}^{T} a_{ij}(t)\left( \hat{x}_{j} - \hat{x}_{i}\right)
\leq a_{ij}(t) \left( \frac{1}{2\theta}\left\| e_{i}\right\|^2 + \frac{\theta}{2} \left\| \hat{x}_{j}- \hat{x}_{i}\right\|^2 \right)
\end{array}
\]
% \end{align*}
with $\theta > 0$,
and
% \begin{align*}
\[
\begin{array}{rl}
 & {\displaystyle \sum_{i=1}^{N}\sum_{j=1}^{N}} a_{ij}(t)\Delta\hat{x}_{i}^{T}\Delta\hat{x}_{j}-{\displaystyle \sum_{i=1}^{N}\sum_{j=1}^{N}} a_{ij}(t)\Delta\hat{x}_{i}^{T}\Delta\hat{x}_{i}\\
= & -\frac{1}{2}{\displaystyle \sum_{i=1}^{N}\sum_{j=1}^{N}} a_{ij}(t)\left(\Delta\hat{x}_{i}^{T}\Delta\hat{x}_{i}-2\Delta\hat{x}_{i}^{T}\Delta\hat{x}_{j}+ \Delta \hat{x}_{j}^{T}\Delta\hat{x}_{j}\right)\\
= & -\frac{1}{2}{\displaystyle \sum_{i=1}^{N}\sum_{j=1}^{N}} a_{ij}(t)\left\| \hat{x}_{j}- \hat{x}_{i}\right\|^{2}
\end{array}
\]
% \end{align*}
where the first equality holds because the graph $\mathcal{G}(t)$ is
balanced. Moreover, by Cauchy\textendash Schwarz inequality, one has
% \begin{align*}
\[
\bigg\| \sum_{j=1}^{N} a_{ij}(t)\left( \hat{x}_{j}- \hat{x}_{i}\right) \bigg\| ^{2}
\leq d_{in}^{i}(t) \sum_{j=1}^{N}a_{ij}(t)\left\| \hat{x}_{j} - \hat{x}_{i}\right\|^{2}.
\]
% \end{align*}
Hence, it can be further obtained that 
% \begin{align*}
\[
\begin{array}{rl}
\dot{V}\le & -\frac{\beta}{2} \displaystyle \sum_{i=1}^{N} \sum_{j=1}^{N} a_{ij}(t) \left( -\frac{\left\| e_{i}\right\|^2}{\theta} + (1 - 2 |\nu_i| \beta d_{in}^{i}(t) -\theta) \right.\\
 & \left. \cdot \left\| \hat{x}_{j}- \hat{x}_{i}\right\|^{2}\vphantom{\displaystyle \sum_{j=1}^{N}} \right).
\end{array}
\]
% \end{align*}
A sufficient condition to ensure $\dot{V} \leq 0$ is
% \begin{align*}
\[
\left\| e_{i}\right\|^2 \leq \frac{\theta}{d^{i}_{in}(t)}\left(1 - 2|\nu_i| \beta d_{in}^{i}(t) -\theta \right)\sum_{j=1}^{N} a_{ij}(t) \left\|\hat{x}_{j} - \hat{x}_{i}\right\|^{2}.
\]
% \end{align*}
The right hand side of the above inequality obtains the maximum value when $\theta = \frac{1}{2} - |\nu_i| \beta d_{in}^{i}(t)$. Define $c_i \in (0,1)$, then the triggering condition \eqref{eq: event-driven condition} guarantees $\dot{V} \leq 0$.
Define the domain $\Omega_{0} = \{(x, \lambda)|V(x,\lambda) \leq V(x(0), \lambda(0))\}$ with $x = (x_1^T, \ldots, x_N^T)^T$, $\lambda = (\lambda_1^T, \ldots, \lambda_N^T)^T $. Because $\dot{V} \leq 0$, it is clear that all system trajectories are bounded and contained within the domain $\Omega_0$. Next, define the domain $\Omega_{e} = \{(x, \lambda)|\dot{V} = 0\}$. It is clear that $\|\dot{x}_i\|$ and $\|\dot{\lambda}_i\|$ are bounded for any bounded $x_i$, $\lambda_i$, $\forall i \in \mathcal{N}$.
Invoking the Invariance Principle \cite[Theorem 2.3]{barkana2014defending}, all limit points of the bounded trajectory belong to the domain $\Omega_f = \Omega_0 \cap \Omega_{e}$, which implies $\lim_{t \rightarrow \infty} \|e_i(t)\| = 0$ and $\lim_{t \rightarrow \infty} \hat{x}_i = \lim_{t \rightarrow \infty} \hat{x}_j$. It follows that $\lim_{t \rightarrow \infty} x_i = \lim_{t \rightarrow \infty} x_j$. Summing up \eqref{eq: algorithm individual agents dynamics} similarly to \eqref{eq: summing up} we obtain $\lim_{t \rightarrow \infty} (x_i,\lambda_i ) = (x_i^*, \lambda_i^*)$, which completes the proof.

% {\color{blue}Since $V \geq 0$ and $V$ is continuous with respect to $t$, one has that $\lim_{t \rightarrow \infty} V$ exists and is finite, and $\lim_{t \rightarrow \infty} \dot{V} = 0$, which implies that $\lim_{t \rightarrow \infty} \|e_i(t)\| = 0$ and $\lim_{t \rightarrow \infty} \hat{x}_i = \lim_{t \rightarrow \infty} \hat{x}_j$, $\forall i, j$. Combining these two conditions, we obtain $\lim_{t \rightarrow \infty} {x}_i = \lim_{t \rightarrow \infty} {x}_j$, $\forall i, j$. Consequently, $\lim_{t \rightarrow \infty} \dot{\lambda}_i = \mathbf{0}$, $\lim_{t \rightarrow \infty} \dot{x}_i = \mathbf{0}$ due to the algorithmic dynamics \eqref{eq: algorithm individual agents}, \eqref{eq: input with event-triggered}, then the states reach the equilibrium point $\lim_{t \rightarrow \infty} (x_i,\lambda_i) = (x_i^*, \lambda_i^*)$, which completes the proof.}
\end{proof}
\begin{remark}\label{Remark: event}
To avoid the Zeno behavior in practice, one can implement the following triggering condition instead,
% \begin{align*}
\[
\begin{array}{rl}
\left\|e_{i}(t)\right\|^2 \geq \max \left\{\frac{c_i \left(\frac{1}{2}-\beta |\nu_i| \right)^2}{d^{i}_{in}(t)} \displaystyle \sum_{j=1}^{N} a_{ij}(t)\left\| \hat{x}_{j}-\hat{x}_{i}\right\|^{2},\zeta\right\}
\end{array}
\]
% \end{align*}
where $\zeta > 0$ is an small predefined error. 
However, under this condition no exact consensus but only practical consensus of $x_i$, $i\in \mathcal{N}$ can be reached since this condition only guarantees Lyapunov boundedness, and the closer $\zeta$ gets to zero, the more accurate the result can be.
Nevertheless, we will see in the next section that the event-triggered scheme \eqref{eq: event-driven condition} can be readily applied to the discretized algorithm, wherein Zeno behavior is avoided naturally.
\end{remark}
\section{Discrete-time Algorithm}\label{Discrete-time Algorithm}
In this section, we study the discretization of the continuous-time algorithm \eqref{eq: algorithm individual agents}. By applying the forward Euler method to algorithm \eqref{eq: algorithm individual agents} with respect to a constant stepsize $\delta > 0$, we can obtain the following discrete-time algorithm
\begin{subequations}\label{eq: algorithm discrete-time}
\begin{align}
{x}_i(k+1) = & {x}_i(k) - \delta \left(\alpha \nabla f_i({x}_i(k)) + {\lambda}_i(k)\right) \label{eq: algorithm discrete-time part 1}\\
{\lambda}_i(k+1) = & {\lambda}_i(k) - \delta u_i (k) \label{eq: algorithm discrete-time part 2}
\end{align}
\end{subequations}
where ${u}_i(k)$ is the input taking the diffusive couplings of $x_i(k)$, i.e.,
% \begin{align}\label{eq: input u_i discrete-time}
\[
u_i(k) = \beta \displaystyle \sum_{j=1}^{N} a_{ij}(k) \left( {x}_j(k) - {x}_i(k) \right). \numberthis \label{eq: input u_i discrete-time}
\]
% \end{align}

% \vspace{-1cm}
\subsection{IFP Preservation}
% It is known that the Euler discretization of an exponentially stable dynamical system can achieve convergence given a sufficiently small stepsize \cite{qu2018exponential,stetter1973analysis}.
% Nevertheless, to ensure convergence under uniformly jointly strongly connected digraphs, 
We analyze the discrete-time algorithm from the perspective of passivity in this subsection.

\begin{lemma}[IFP preservation]\label{Lemma IFP preservation}
By selecting an appropriate stepsize 
\begin{align}\label{eq: stepsize}
\delta < \frac{1}{\alpha} \frac{4 l_i - 2 \mu_i}{2 l_i^2 - \mu_i^2}, \quad \forall i \in \mathcal{N},
\end{align}
system \eqref{eq: algorithm discrete-time} is IFP($\tilde{\nu}_i$) from ${u}_i$ to $\Delta {x}_i$ with
\begin{align}\label{eq: IFP index discrete-time}
\tilde{\nu}_i = - \frac{ \left(\frac{1}{\alpha \mu_i} + \delta \left( \frac{1}{2} + \frac{l_i}{\mu_i} \right) \right)^2}{\alpha \delta \left(\frac{\mu_i}{2} - \frac{l_i^2}{\mu_i}\right) + \frac{2 l_i}{\mu_i} - 1}.
\end{align}
% $\tilde{\nu}_i = - \max_{\mu_i \leq s \leq l_i} \frac{\left( \left(\frac{2}{\alpha \mu_i} + \delta \right) - \frac{2\delta}{\mu_i} s \right)^2 }{ 4 \alpha \delta\left( \frac{\mu_i}{2} - \frac{s^2}{\mu_i}\right) + 4\left(\frac{2 s}{\mu_i} - 1 \right)}$.
\end{lemma}
\begin{proof}
Denote ${z}_i(k) = \alpha \nabla {f}_i({x}_i(k)) + {\lambda}_i(k)$. Adopt the storage function $\tilde{V}_i = \frac{1}{\delta} V_i$, where $V_i = $
First, we have
\begin{subequations}
\begin{align*}
& {V}_i(k+1) - {V}_i(k)\\
= & \frac{1}{\alpha \mu_i} \left(\|z_i(k+1)\|^2 - \|z_i(k)\|^2 \right) \numberthis \label{eq: storage function discrete-time part 1}\\
& - \Delta x_i(k+1)^T \Delta \lambda_i(k+1) + \Delta x_i(k)^T \Delta \lambda_i(k) \numberthis \label{eq: storage function discrete-time part 2}\\
& + \alpha \left(f_i(x_i(k))- f_i(x_i(k+1)) + \nabla f_i(x^*_i)^T \cdot \right.\\
& \left. ( x_i(k+1) - x_i(k)) \right) \numberthis \label{eq: storage function discrete-time part 3}
\end{align*}
\end{subequations}
Denote $\nabla f_i(x_i(k+1)) - \nabla f_i(x_i(k)) = B_{x_i^{+}} (x_i(k+1) - x_i(k))$, where $B_{x_i^{+}}$ is definied similarly as $B_{x_i}$ in \eqref{eq: Bx}, and is a positive definite matrix satisfying $\mu_i I \leq B_{x_i^{+}} \leq l_i I$. Then, \eqref{eq: storage function discrete-time part 1} becomes
% \begin{align*}
\[
\begin{array}{rl}
& \frac{1}{\alpha \mu_i} \left( \| \alpha \nabla f_i(x_i(k+1)) + \lambda_i(k+1) \|^2 - \|z_i(k)\|^2 \right)\\
= & \frac{1}{\alpha \mu_i} \left( \| \alpha ( \nabla f_i(x_i(k)) - B_{x_i^{+}} \delta z_i(k) ) + \lambda_i(k) - \delta u_i(k) \|^2 \right. \\
& \left. - \|z_i(k)\|^2 \vphantom{B_{x_i^{+}}} \right)\\
= & \frac{1}{\alpha \mu_i} \left( \left\| \left( I - \alpha \delta B_{x_i^{+}} \right) z_i(k) - \delta u_i(k) \right \|^2 - \| z_i(k) \|^2 \right)\\
= & z_i(k)^T \left( \frac{\alpha \delta^2}{\mu_i} B_{x_i^{+}}^2 - \frac{2\delta}{\mu_i} B_{x_i^{+}} \right) z_i(k) + \frac{\delta^2}{\alpha \mu_i} \|u_i(k)\|^2\\
& - z_i(k)^T \left( \frac{2 \delta}{\alpha \mu_i}I - \frac{2 \delta^2}{\mu_i} B_{x_i^{+}} \right) u_i (k)
\end{array}
\]
% \end{align*}
where the first and second qualities follow from the definition of $z_i(k)$.
By substituting \eqref{eq: algorithm discrete-time} into \eqref{eq: storage function discrete-time part 2}, it can be readily obtained that \eqref{eq: storage function discrete-time part 2} equals to
\begin{align*}
% \[
% \begin{array}{rl}
& \Delta x_i(k)^T \Delta \lambda_i(k) - \left( \Delta x_i(k) - \delta z_i(k) \right)^T \left( \Delta \lambda_i(k) - \delta u_i (k)\right)\\
= & \delta z_i(k)^T \Delta \lambda_i(k) - \delta^2 z_i(k)^T u_i(k) + \delta \Delta x_i(k)^T u_i(k).
% \end{array}
% \]
\end{align*}
Moreover, by the strong convexity of $f_i$, \eqref{eq: storage function discrete-time part 3} satisfies
% \begin{align*}
\[
\begin{array}{rl}
& \alpha \left(f_i(x_i(k))- f_i(x_i(k+1)) + \nabla f_i(x^*_i)^T ( - \delta z_i(k)) \right)\\
\leq & \alpha \left( \vphantom{\frac{\mu_i}{2}} -\nabla f_i(x_i(k))^T \left( x_i(k+1) - x_i(k) \right)\right.\\
& \left. - \frac{\mu_i}{2} \|x_i(k+1) - x_i(k)\|^2
- \delta \nabla f_i(x^*_i)^T z_i(k) \right)\\
= & \alpha \delta \left(\nabla f_i(x_i(k)) - \nabla f_i(x_i^*) \right)^T z_i(k) - \frac{\alpha \mu_i}{2} \delta^2 \|z_i(k)\|^2
\end{array}
\]
% \end{align*}
where the equality follows from \eqref{eq: algorithm discrete-time part 1}. Summing up \eqref{eq: storage function discrete-time part 1}, \eqref{eq: storage function discrete-time part 2} and \eqref{eq: storage function discrete-time part 3} and by the definition of $\tilde{V}_i$, we obtain
% \begin{align*}
\[
\begin{array}{rl}
& \tilde{V}_i(k+1) - \tilde{V}_i(k)\\
\leq & z_i(k)^T \left( \frac{\alpha \delta}{\mu_i} B_{x_i^{+}}^2 - \frac{2}{\mu_i} B_{x_i^{+}} + \left( 1- \frac{\alpha \delta \mu_i }{2}\right)I \right) z_i(k) \\
& -z_i(k)^T \left( \left(\frac{2}{\alpha \mu_i} + \delta \right) I - \frac{2 \delta}{ \mu_i} B_{x_i^{+}} \right) u_i + \Delta x_i(k)^T u_i(k)\\
& + \frac{\delta}{\alpha \mu_i} \| u_i(k) \|^2\\
= & \left(\begin{smallmatrix}
z_i(k)\\ u_i(k)
\end{smallmatrix} \right)^T M \left(\begin{smallmatrix}
z_i(k)\\ u_i(k)
\end{smallmatrix} \right) + \Delta x_i(k)^T u_i(k) - \tilde{\nu}_i \| u_i(k)\|^2
\end{array}
\]
% \end{align*}
where $M = \left(\begin{smallmatrix}
\frac{\alpha \delta}{\mu_i} B_{x_i^{+}}^2 - \frac{2}{\mu_i} B_{x_i^{+}} + \left( 1 - \frac{\alpha \delta \mu_i }{2}\right)I & \left(\frac{1}{\alpha \mu_i} + \frac{\delta}{2} \right) I - \frac{ \delta}{ \mu_i} B_{x_i^{+}}\\
* & \tilde{\nu}_i I 
\end{smallmatrix} \right)$, and the inequality follows from the definition of $z_i(k)$.
% By the Schur complement along with the eigenvalue decomposition of $B_{x_i^{+}}$, one has $M \leq 0$ if $\delta < \frac{1}{\alpha} \frac{4 l_i - 2 \mu_i}{2 l_i^2 - \mu_i^2}$ and 
% \begin{align}\label{eq: IFP index inequality discrete-time}
% \tilde{\nu}_i \leq & - \max_{\mu_i \leq s \leq l_i} \frac{\left( \left(\frac{2}{\alpha \mu_i} + \delta \right) - \frac{2\delta}{\mu_i} s \right)^2 }{ 4 \alpha \delta\left( \frac{\mu_i}{2} - \frac{s^2}{\mu_i}\right) + 4 \left(\frac{2 s}{\mu_i} - 1 \right)}
% \end{align}
% where $s \in [\mu_i, l_i]$ is a variable.
% Thus, with a proper stepsize $\delta$, we obtain $\tilde{V}_i(k+1) - \tilde{V}_i(k) \leq \Delta x_i^T u_i - \tilde{\nu}_i \| u_i \|^2$, where $\tilde{\nu}_i = - \max_{\mu_i \leq s \leq l_i} \frac{\left( (\frac{2}{\alpha \mu_i} + \delta) - \frac{2\delta}{\mu_i} s \right)^2 }{4 \alpha \delta \left( \frac{\mu_i}{2} - \frac{s^2}{\mu_i}\right) + 4 \left(\frac{2 s}{\mu_i} - 1 \right)}$.
Apparently, since $\mu_i I \leq B_{x_i^{+}} \leq l_i I$, one has
\begin{align*}
% $
M \leq 
\left(
\begin{smallmatrix}
\frac{\alpha \delta l_i^2}{\mu_i} - \frac{2 l_i}{\mu_i} + 1 - \frac{\alpha \delta \mu_i}{2} & \frac{1}{\alpha \mu_i} + \frac{\delta}{2} + \frac{\delta l_i}{\mu_i}\\
* & \tilde{\nu}_i
\end{smallmatrix}
\right)
\otimes I.
\end{align*}
Then, $M \leq 0$ if $\delta < \frac{1}{\alpha} \frac{4 l_i - 2 \mu_i}{2 l_i^2 - \mu_i^2}$ and $\tilde{\nu}_i \leq - \frac{ \left(\frac{1}{\alpha \mu_i} + \delta \left( \frac{1}{2} + \frac{l_i}{\mu_i} \right) \right)^2}{\alpha \delta \left(\frac{\mu_i}{2} - \frac{l_i^2}{\mu_i}\right) + \frac{2 l_i}{\mu_i} - 1}$. Thus, given an appropriate constant $\delta$, we obtain $\tilde{V}_i(k+1) - \tilde{V}_i(k) \leq \Delta x_i^T u_i - \tilde{\nu}_i \| u_i \|^2$, which completes the proof.
\end{proof}
\begin{remark}
The stepsize $\delta$ obtained in condition \eqref{eq: stepsize} is constant and non-diminishing. It is independent of network size and thus less conservative compared to many works in the literature \cite{nedic2017achieving,xie2018distributed,li2018distributed,scutari2019distributed}. Moreover, it can also be easily estimated distributedly given the bounds of indices $\mu_i$ and $l_i$, $\forall i \in \mathcal{N}$.
% The problem in \eqref{eq: IFP index inequality discrete-time} is of scale size and thus easy to solve, given specific parameters.
It can be observed that \eqref{eq: IFP index discrete-time} characterizes the passivity degradation over discretization. If the stepsize $\delta$ is infinitely small, then
$
\lim_{\delta \rightarrow 0^{+}} M =
\left(
\begin{smallmatrix}
I - \frac{2}{\mu_i} B_{x_i^{+}} & {\alpha \mu_i} I\\
* & \tilde{\nu}_i I 
\end{smallmatrix}
\right)
\leq 
\left(
\begin{smallmatrix}
-I & {\alpha \mu_i} I\\
* & \tilde{\nu}_i I 
\end{smallmatrix}
\right)
$ and $\lim_{\delta \rightarrow 0^{+}} \tilde{\nu} = - \frac{1}{\alpha^2 \mu_i^2}$,
% \begin{align*}
% \[
% \begin{array}{rl}
% & \lim_{\delta \rightarrow 0^{+}} \left( - \max_{\mu_i \leq s \leq l_i} \frac{\left( (\frac{2}{\alpha \mu_i} + \delta) - \frac{2\delta}{\mu_i} s \right)^2 }{4 \alpha \delta\left( \frac{\mu_i}{2} - \frac{s^2}{\mu_i}\right) + 4 \left( \frac{2 s}{\mu_i} - 1\right)} \right)\\
% = & - \max_{\mu_i \leq s \leq l_i} \frac{\frac{4}{\alpha^2\mu_i^2}}{4 \left( \frac{2 s}{\mu_i} - 1\right)} = - \frac{1}{\alpha^2 \mu_i^2} = \nu_i
% \end{array}
% \]
% \end{align*}
which recovers to the IFP index $\nu_i$ for the continuous-time system.
\end{remark}

Next, we study the convergence of the algorithm following similar lines of \Cref{Lemma convergence of IFP}.
\begin{theorem}\label{Theorem Discretization stability}
Under Assumptions \ref{Ass: objective function}--\ref{Ass: graph 2}, the states of algorithm \eqref{eq: algorithm discrete-time}, \eqref{eq: input u_i discrete-time} with initial condition $\sum_{i=1}^{N} \lambda_i(0)= \mathbf{0}$ will converge to the optimal solution to problem \eqref{eq: distributed convex optimization} if the stepsize $\delta < \frac{1}{\alpha} \frac{4 l_i - 2 \mu_i}{2 l_i^2 - \mu_i^2} $, $\forall i \in \mathcal{N}$, and the following condition holds,
% \begin{align}\label{eq: convergence condition discrete-time}
\[
\begin{array}{rl}
|\tilde{\nu}_i| \beta d^{i}_{in}(k) < \frac{1}{2}, ~\forall i \in \mathcal{N}, ~\forall k \geq 0 \numberthis \label{eq: convergence condition discrete-time}
\end{array}
\]
% \end{align}
where $d_{in}^{i}(k)$ denotes the in-degree of the $i$th agent at $k$.
\end{theorem}
The proof is similar to that of \Cref{Lemma convergence of IFP} by considering the discrete-time Lyapunov function candidate $\tilde{V} = \sum_{i = 1}^{N} \tilde{V}_i$ where $\tilde{V}_i$ is the storage function defined in \Cref{Lemma IFP preservation}.

\subsection{Discrete-time Event-triggered Mechanism}
Similarly, let us consider the discrete-time algorithm incorporating the same event-triggered mechanism, i.e., replacing \eqref{eq: input u_i discrete-time} with
\begin{equation}\label{eq: input with event-triggered discrete-time}
\begin{aligned}
% \[
% \begin{array}{rl}
u_i(k) = & \beta \displaystyle \sum_{j=1}^{N} a_{ij}(k) \left( \hat{{x}}_j(k) - \hat{{x}}_i(k) \right)
% \end{array}
% \]
\end{aligned}
\end{equation}
where $\hat{x}_i(k)$, $i \in \mathcal{N}$ denotes the last state of agent $i$ sent to its neighbors and and $\hat{x}_i(0) := x_i(0)$. The triggering condition is
% \begin{align*}\label{eq: event-driven condition discrete-time}
\[
\begin{array}{rl}
\left\|{e}_{i}(k) \right\|^2
\geq &
\frac{c_i \left(\frac{1}{2} - |\tilde{\nu}_i| \beta d^{i}_{in}(k)\right)^2}{d^i_{in}(k)} \displaystyle \sum_{j=1}^{N} a_{ij}(k)\left\| \hat{{x}}_{j}(k)-\hat{{x}}_{i}(k)\right\|^{2} \numberthis \label{eq: event-driven condition discrete-time}
\end{array}
\]
% \end{align*}
where ${e}_{i}(k) = {x}_{i}(k) - \hat{x}_{i}(k)$ and $ c_i\in (0,1)$.
Then we have the following theorem on the convergence of discrete-time algorithm under event-triggered communication.

\begin{theorem}\label{Theorem convergence under ETC discrete-time}
Under Assumptions \ref{Ass: objective function}--\ref{Ass: graph 2}, if the stepsize satisfies $\delta < \frac{1}{\alpha} \frac{4 l_i - 2 \mu_i}{2 l_i^2 - \mu_i^2} $, $\forall i \in \mathcal{N}$, $\alpha,\beta$ are designed such that \eqref{eq: convergence condition discrete-time} holds, and the triggering instant for agent $i$, $i\in \mathcal{N}$ to transmit its current information of ${x}_i$ is chosen whenever $d^i_{in}(k) > 0$ and triggering condition \eqref{eq: event-driven condition discrete-time} is satisfied.
Then the states of algorithm
\eqref{eq: algorithm discrete-time}, \eqref{eq: input with event-triggered discrete-time} with initial condition $\sum_{i=1}^{N} {\lambda}_{i}(0)=\mathbf{0}$ will converge to the optimal solution to problem \eqref{eq: distributed convex optimization}.
\end{theorem}
The proof follows from arguments similar to that of \Cref{convergence under ETC continuous-time}, and is omitted here.
\begin{remark}
\Cref{Theorem convergence under ETC discrete-time} compares favorably to other works in the literature \cite{kia2015distributed,liu2016event,kajiyama2018distributed,liu2019distributed}, where only undirected or fixed topologies are considered.
Though we do not derive conditions that excludes the case where the states might update at every time step, it is shown with an example in \Cref{Numerical Example} that communication is greatly reduced using the proposed mechanism with appropriate parameters.
\end{remark}

\section{Numerical Example}\label{Numerical Example}
In this section, we provide a numerical example to illustrate the proposed continuous-time and discrete-time algorithms under event-triggered communication.
Consider the distributed optimization problem \eqref{eq: distributed convex optimization} among $5$ agents over a weight-balanced and uniformly jointly strongly connected digraph that is switching every two seconds among two modes, as shown in \Cref{fig: Connecting graphs}. The weights are set to $a_{ij} \in \{0, 1\}$ for simplicity.
\begin{figure}[htbp]%[t]
\centering
\includegraphics[width= 0.8\linewidth]{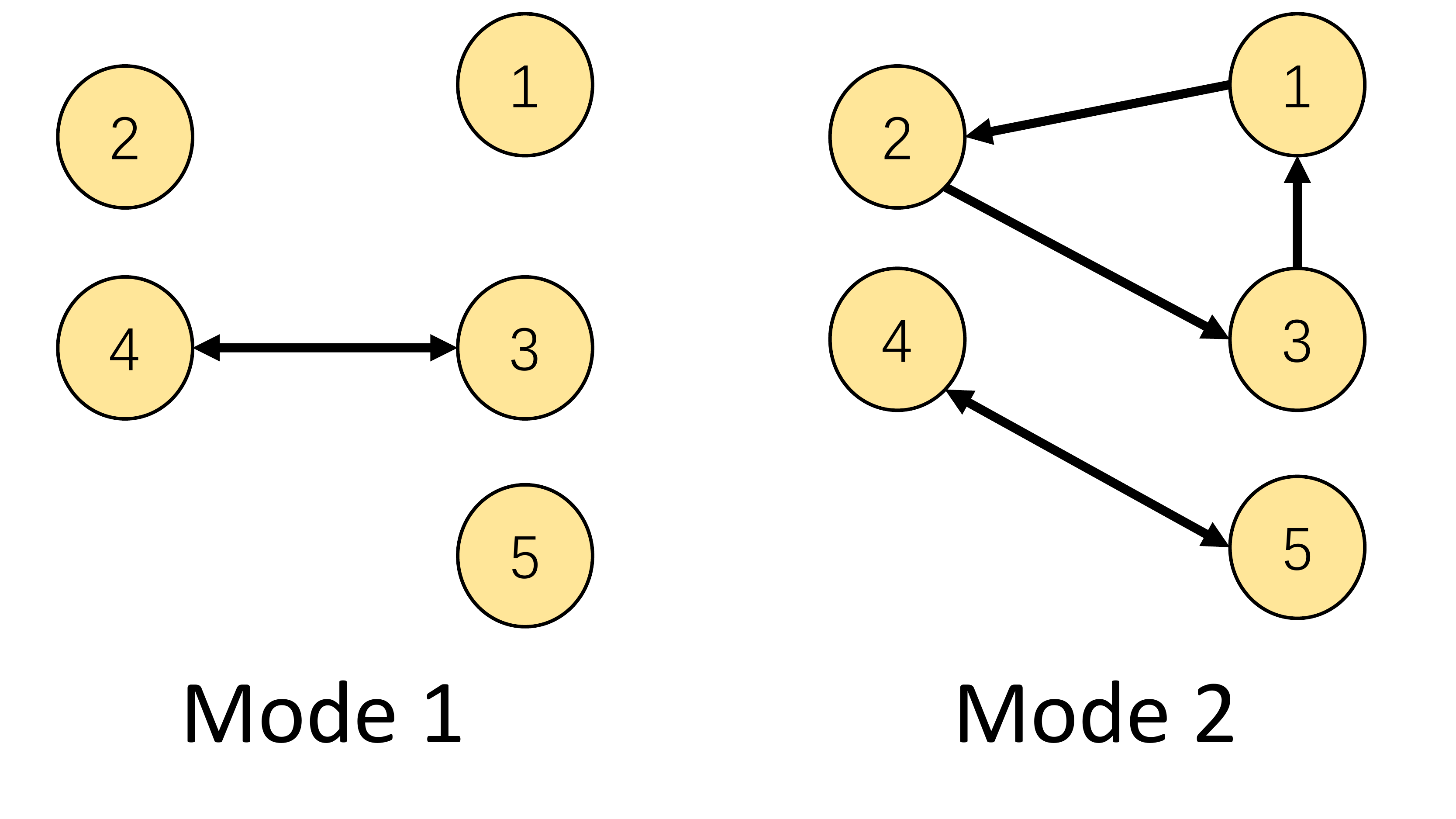}
\caption{The communication graph is weight-balanced and switching every two seconds among the two modes.}
\label{fig: Connecting graphs}
\end{figure}
The local objective functions are
% \begin{align*}
\[
\begin{array}{rl}
f_1(\mathrm{x}) = & \frac{1}{2}\mathrm{x}^2+3\mathrm{x}+1, ~~ f_2(\mathrm{x})=\frac{1}{2}\mathrm{x}^2 - \mathrm{x},\\
f_3(\mathrm{x}) = & \mathrm{x}^2 + \sin{\mathrm{x}}, ~~f_4(\mathrm{x})=\ln(e^{2\mathrm{x}}+1) + 0.5 \mathrm{x}^2,\\
f_5(\mathrm{x}) = & \ln(e^{2\mathrm{x}} + e^{-0.2\mathrm{x}}) + 0.6\mathrm{x}^2.
\end{array} 
\]
% \end{align*}
We obtain that these functions are strongly convex with $\mu_1 = \mu_2 = \mu_3 = \mu_4 = 1$, $\mu_5 = 1.2$ and have Lipschitz gradient with $l_1 = l_2 = 1$, $l_3 = 3$, $l_4 = 2$ and $l_5 = 2.41$. Let $\alpha = 1$, $c_i = 0.99$, initial conditions $x_i(0) \in [0,1]$, $\lambda_i(0) = 0$, and we consider the following two cases.\\
\textbf{Continuous-time}:
We obtain $\nu_i = - \frac{1}{\mu_i^2}$ from \Cref{Lemma error system IFP}, and $0 < \beta < 0.5$ from \Cref{Lemma convergence of IFP}.
Choose $\beta = 0.2$ and apply the continuous-time algorithm \eqref{eq: algorithm individual agents} under event-triggered control laws \eqref{eq: input with event-triggered}, \eqref{eq: event-driven condition} in MATLAB.
The trajectories of $x_i(t)$ and triggering instants of $x_i$ under event-triggered communication are shown in \Cref{fig: Trajectories of continuous-time algorithm}. It can be observed that the states converge to the optimal solution while communication is reduced due to both the jointly strongly connected graph and the event-triggered mechanism.\\
\textbf{Discrete-time}:
We obtain from \Cref{Lemma IFP preservation} that $\delta < 0.59$. Select $\delta = 0.1$, then $\tilde{\nu}_1 = \tilde{\nu}_2 = -1.39$, $\tilde{\nu}_3 = -0.44$, $\tilde{\nu}_4 = -0.59$, and $\tilde{\nu}_5 -0.45$. $\beta$ should be less than $0.36$ according to \Cref{Theorem Discretization stability}.
The convergence results of the discrete-time algorithm \eqref{eq: algorithm discrete-time} under event-triggered control laws \eqref{eq: input with event-triggered discrete-time}, \eqref{eq: event-driven condition discrete-time} with $\beta = 0.1, 0.3$ are shown in \Cref{fig: Trajectories of discrete-time algorithm,fig: triggering instant discrete-time}, respectively.
It can be observed that communication is greatly reduced here.
We can see from \eqref{eq: event-driven condition discrete-time} that if other parameters are fixed, the events are triggered more frequently under a larger $\beta$. 
Meanwhile, $\beta$ denotes the coupling strength of agents' states, which affects consensus speed.
Thus, there exists a trade-off between the triggering instant and convergence performance, which can be observed from \Cref{fig: performance of discrete-time algorithm}.

\begin{figure}%[H]
\centering
\subfigure[The trajectories of $x_i(t)$.]{\includegraphics[width = 0.48\linewidth]{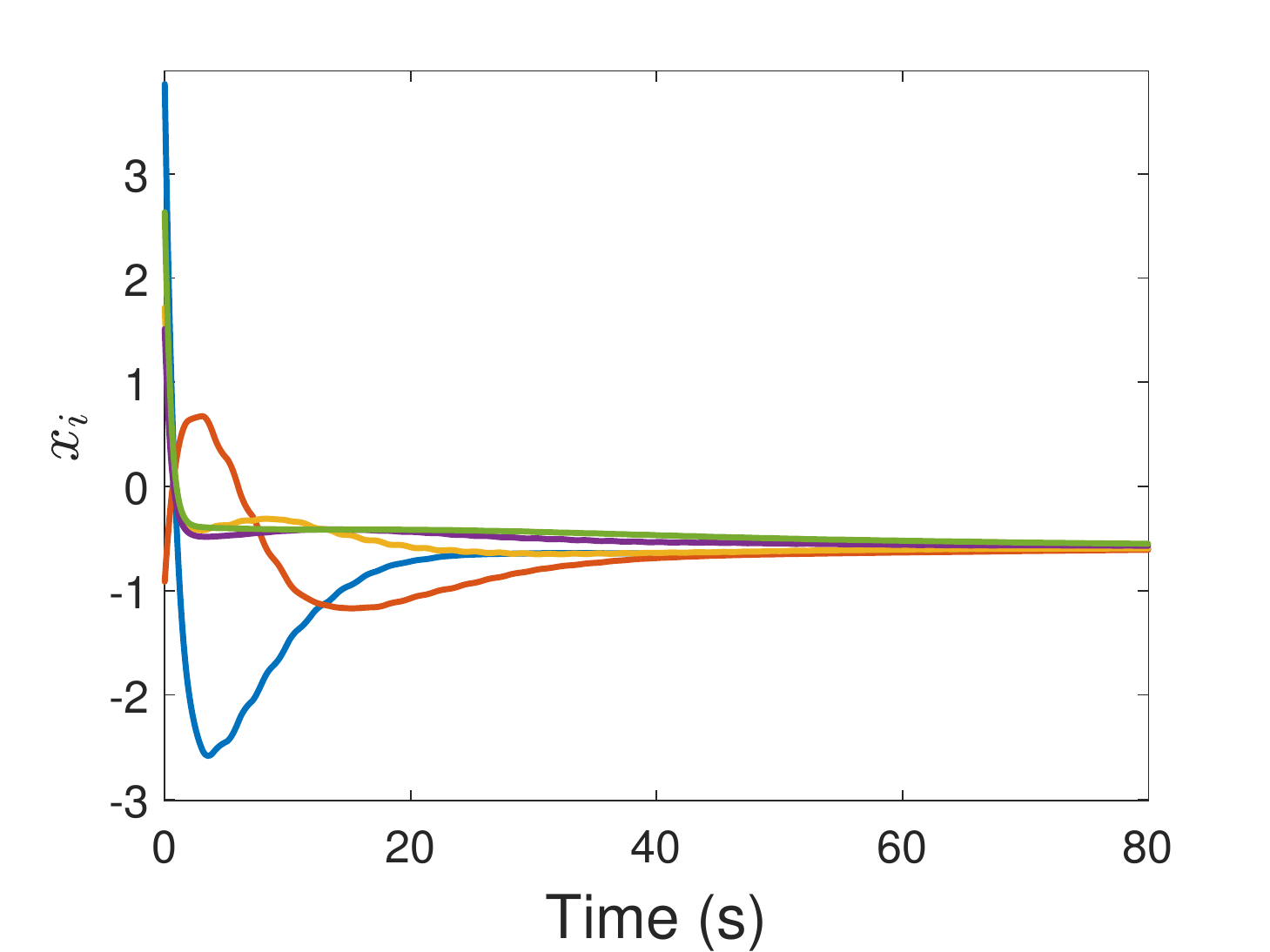}}
\subfigure[triggering instant of ${x}_i$.]{\includegraphics[width =0.48 \linewidth]{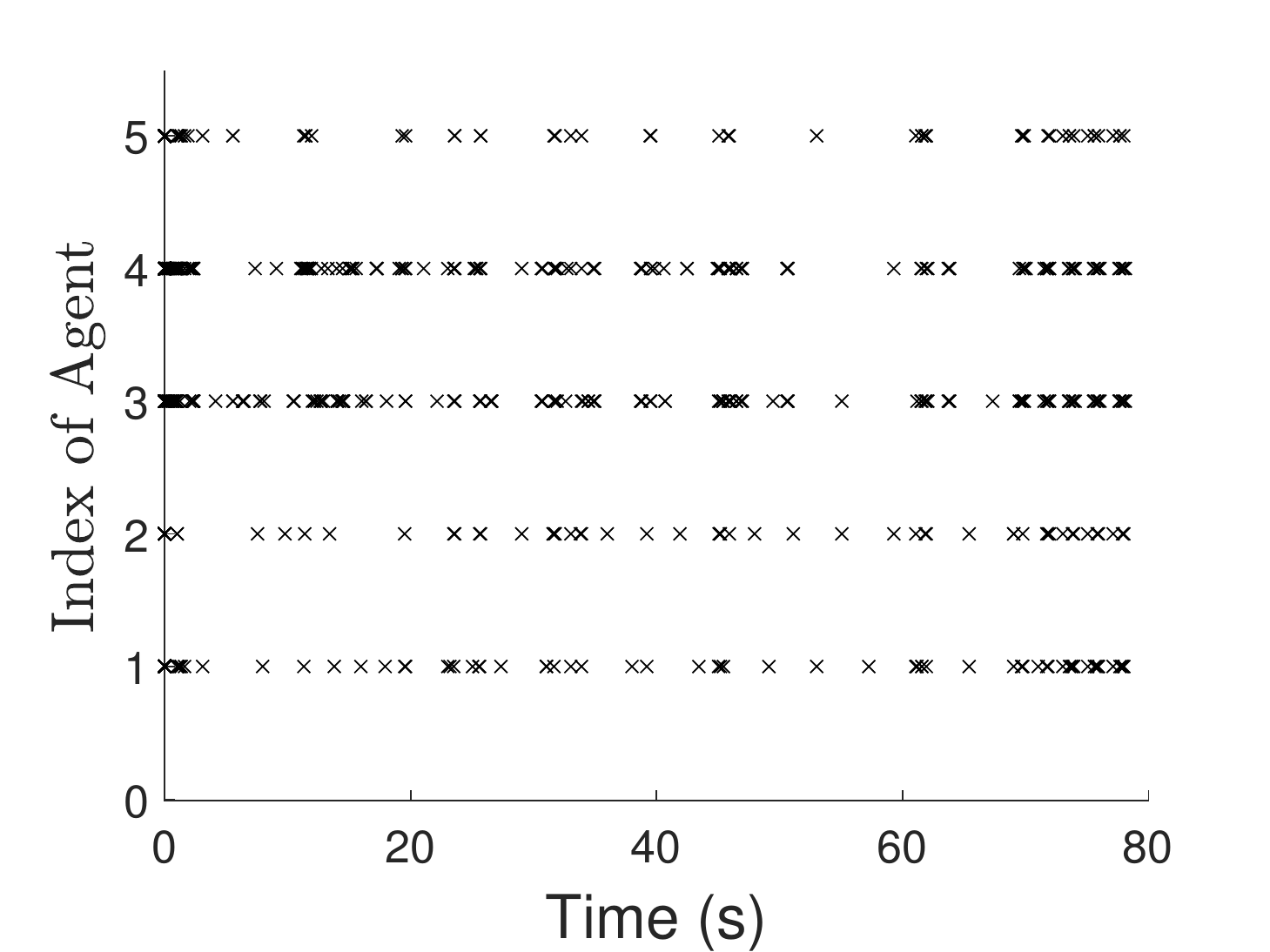}}

\caption{The trajectories of $x_i(t)$ and triggering instant of ${x}_i$ with $\beta = 0.2$ for the continuous-time algorithm under event-triggered communication.}
\label{fig: Trajectories of continuous-time algorithm}
\end{figure}

\begin{figure}[htbp]
\centering
\subfigure[The trajectories of $x_i(k)$.]{\includegraphics[width =1\linewidth]{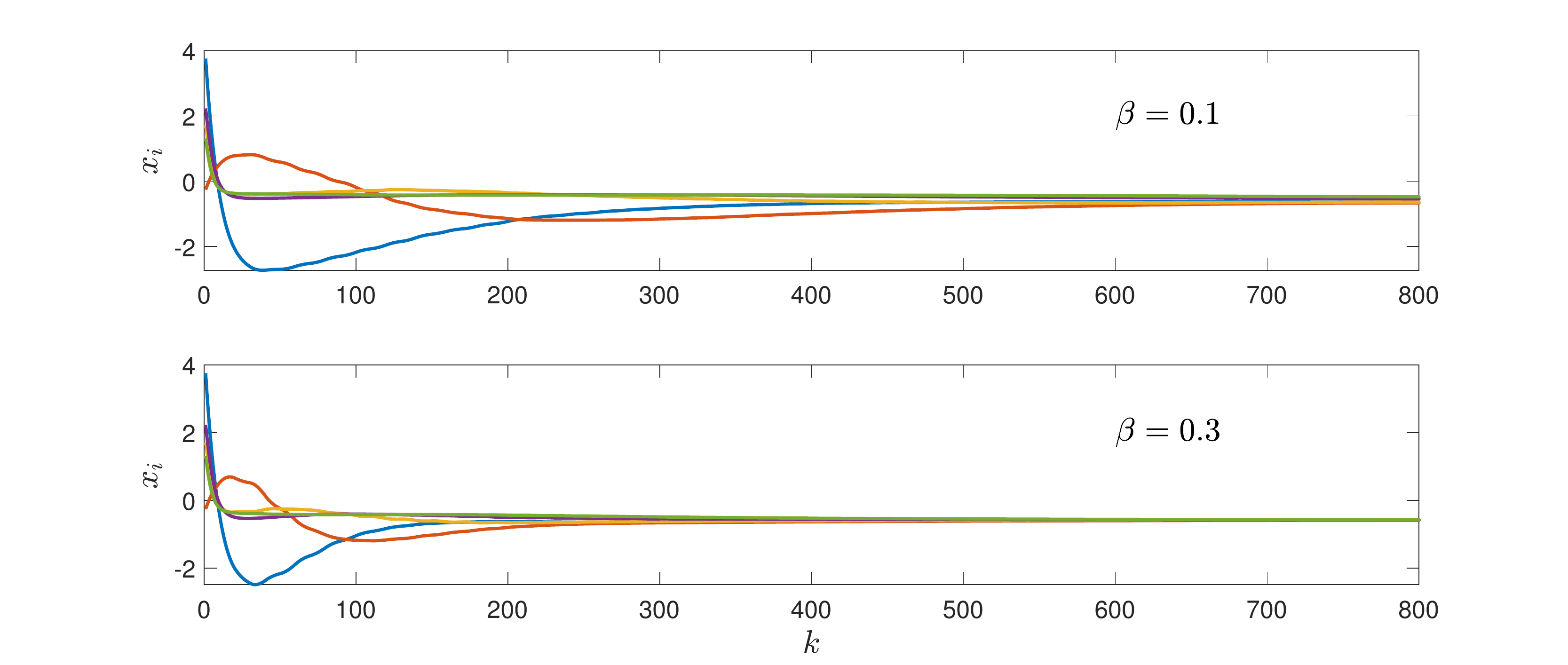}\label{fig: Trajectories of discrete-time algorithm}}
\subfigure[Triggering instant of ${x}_i$ with $\beta = 0.1$ in the upper and $\beta = 0.3$ in the lower.]{\includegraphics[width = 1 \linewidth]{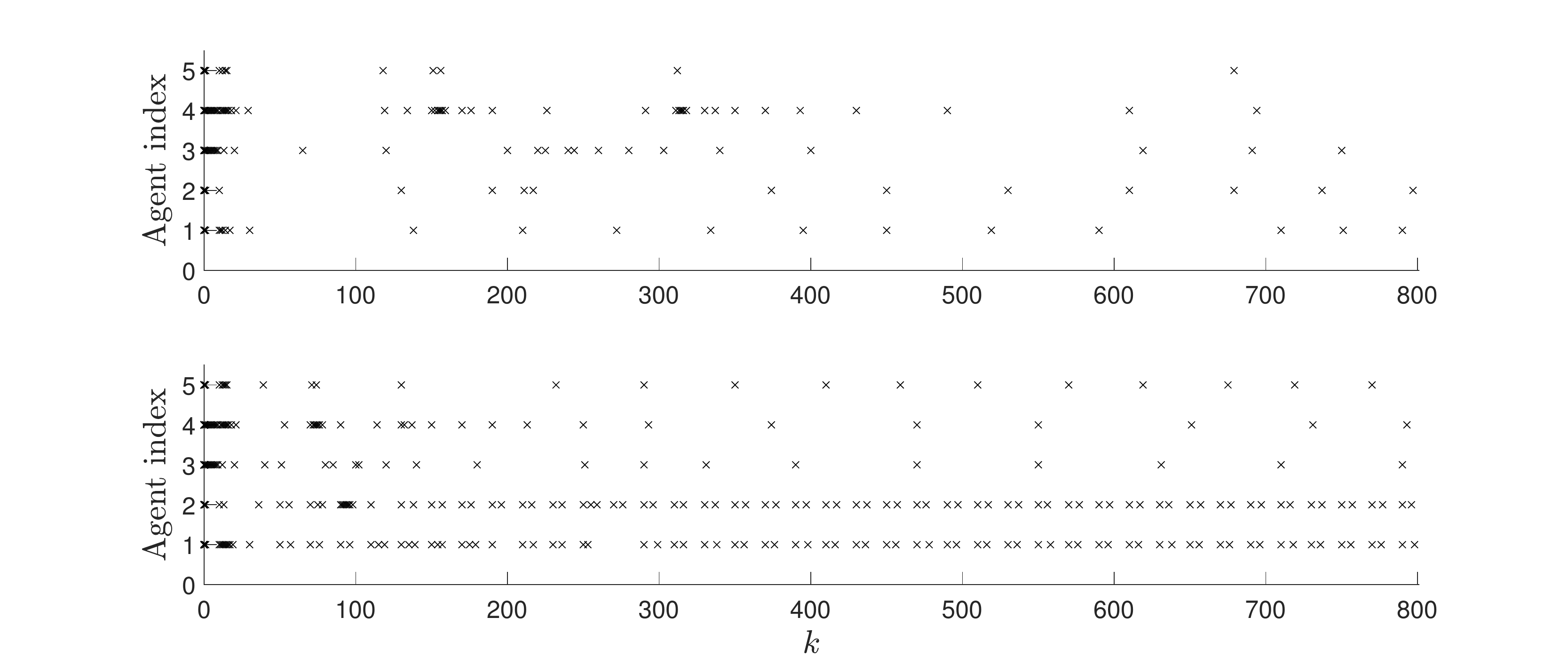}\label{fig: triggering instant discrete-time}}
\caption{The trajectories of $x_i(k)$ and triggering instant of ${x}_i$ for the discrete-time algorithm under event-triggered communication.}
\label{fig: performance of discrete-time algorithm}
\end{figure}

\section{Conclusion}\label{Conclusion}
We have proposed an event-trigger communication mechanism for the distributed continuous-time algorithm and its discrete-time counterpart over uniformly jointly strongly connected balanced graphs via the property of IFP.

% \section*{Acknowledgment}
% Fundings 
% The authors would like to thank the editors and reviewers
% for their valuable comments.

\bibliographystyle{IEEEtran}
\bibliography{References}
\end{document}